\documentclass[a4paper,11pt,oneside,headsepline,abstracton]{scrartcl}

\usepackage{ifdraft}

\usepackage[utf8]{inputenc}
\usepackage[T1]{fontenc}

\usepackage{lmodern}
\usepackage{fourier}

\usepackage{amssymb, amsfonts}

\usepackage{mathrsfs} 
\usepackage{dsfont}
\usepackage{xcolor}
\usepackage{lettrine}
\usepackage{url}

\usepackage{stmaryrd}

\usepackage[english]{babel}

\usepackage{scrpage2}
\usepackage[margin=7em]{geometry}
\usepackage[compact]{titlesec}
\usepackage{footmisc}
\usepackage{needspace}

\usepackage{amsmath}
\usepackage[thmmarks, amsmath, amsthm]{ntheorem}

\usepackage[style=alphabetic,firstinits,url=false,sortcites=true,maxbibnames=99,backend=bibtex]{biblatex}
\bibliography{Lit.bib}

\ifdraft{{}}{\usepackage{hyperref}}

\ifdraft{\usepackage[colorinlistoftodos]{todonotes}}{}

\usepackage{booktabs}
\usepackage{enumitem}
\usepackage{float}

\ifdraft{\date{\today}}{\date{}}

\clearscrheadings
\automark[section]{subsection}

\lohead{\headertitle \;---\, \rightmark}
\rohead{\headerauthors}
\renewcommand{\subsectionmark}[1]{}

\newenvironment{plainfootnotes}{
  \deffootnote[0em]{0em}{0em}{}
}{
  \deffootnote[1em]{1.5em}{1em}{\textsuperscript{\thefootnotemark}}
}

\newif\ifsubsectionstylePrefixParagraph
\newif\ifsubsectionstyleRunin

\subsectionstylePrefixParagraphtrue
\subsectionstyleRunintrue

\titleformat{\section}[hang]
{\Large\sffamily\bfseries}
{\thesection\hspace{0.1em}{|}}{0.25em}{}

\ifsubsectionstyleRunin
\ifsubsectionstylePrefixParagraph

\titleformat{\subsection}[runin]
{\normalfont\bfseries}
{\S\hspace{.1em}\thesubsection}{0.35em}{}[.\hspace*{.5em}]

\else

\titleformat{\subsection}[runin]
{\normalfont\bfseries}
{\thesubsection\hspace{.1em}|}{0.25em}{}[.\hspace*{.5em}]

\fi
\else
\ifsubsectionstylePrefixParagraph

\titleformat{\subsection}[wrap]
{\normalfont\bfseries\selectfont\filright}
{\S\thesubsection}{.35em}{}
\titlespacing{\subsection}
{16pc}{1.5ex plus .1ex minus .2ex}{1pc}

\else

\titleformat{\subsection}[wrap]
{\normalfont\bfseries\selectfont\filright}
{\thesubsection\hspace{.1em}|}{.25em}{}
\titlespacing{\subsection}
{16pc}{1.5ex plus .1ex minus .2ex}{1pc}

\fi
\fi

\ifsubsectionstylePrefixParagraph

\titleformat{\subsubsection}[runin]
{\normalfont\bfseries}
{\S\hspace{.1em}\thesubsubsection}{0.35em}{}[.]

\else

\titleformat{\subsubsection}[runin]
{\normalfont\bfseries}
{\thesubsubsection\hspace{.1em}|}{0.25em}{}[.]
  
\fi

\AtEveryBibitem{\clearfield{isbn}}
\AtEveryBibitem{\clearlist{language}}

\numberwithin{equation}{section}

\newtheorem{theoremcounter}{theoremcounter}[section]

\theoremstyle{plain}

\newtheorem{lemma}[theoremcounter]{Lemma}
\newtheorem{proposition}[theoremcounter]{Proposition}
\newtheorem{theorem}[theoremcounter]{Theorem}

\theoremnumbering{Roman}
\newtheorem{maintheoremcounter}{maintheoremcounter}

\newtheorem{maintheorem}[maintheoremcounter]{Theorem}
\newtheorem{maincorollary}[maintheoremcounter]{Corollary}

\theoremstyle{definition}

\theoremstyle{remark}

\newtheorem{example}[theoremcounter]{Example}

\newtheorem{remark}[theoremcounter]{Remark}
\newtheorem{remarks}[theoremcounter]{Remarks}

\newtheorem*{mainremark}{Remark}
\newtheorem*{mainremarks}{Remarks}

\newtheorem*{remarkcomputation}{Computation}

\newenvironment{mainremarksenumerate}
{
\begin{mainremarks}
\begin{enumerate}[label=(\arabic*), leftmargin=0pt,labelindent=\parindent,itemindent=!]
}{
\end{enumerate}
\end{mainremarks}
}

\let\cal\undefined

\newcommand{\tx}{\ensuremath{\text}}

\newcommand{\tbf}{\bfseries}

\newcommand{\thdash}{\nbd th}

\newcommand{\nbd}{\nobreakdash-\hspace{0pt}}

\newcommand{\cal}{\ensuremath{\mathcal}}
\renewcommand{\frak}{\ensuremath{\mathfrak}}

\newcommand{\cO}{\ensuremath{\cal{O}}}

\newcommand{\frake}{\ensuremath{\frak{e}}}

\newcommand{\frakp}{\ensuremath{\frak{p}}}

\newcommand{\rmt}{\ensuremath{\mathrm{t}}}

\newcommand{\rmJ}{\ensuremath{\mathrm{J}}}

\newcommand{\rmM}{\ensuremath{\mathrm{M}}}

\newcommand{\ra}{\ensuremath{\rightarrow}}

\newcommand{\lra}{\ensuremath{\longrightarrow}}

\newcommand{\amid}{\ensuremath{\mathop{\mid}}}

\newcommand{\ZZ}{\ensuremath{\mathbb{Z}}}
\newcommand{\QQ}{\ensuremath{\mathbb{Q}}}
\newcommand{\RR}{\ensuremath{\mathbb{R}}}
\newcommand{\CC}{\ensuremath{\mathbb{C}}}

\newcommand{\isdiv}{\amid}

\renewcommand{\pmod}[1]{\ensuremath{\;(\mathrm{mod}\, #1)}}

\newcommand{\Sp}[1]{\ensuremath{\mathrm{Sp}_{#1}}}

\newcommand{\T}{\ensuremath{\rmt}}
\newcommand{\rT}{\ensuremath{\,{}^\T\!}}

\newcommand{\tr}{\ensuremath{\mathrm{tr}}}

\newcommand{\diag}{\ensuremath{\mathrm{diag}}}

\newcommand{\slashdiv}{\ensuremath{\mathop{/}}}

\newcommand{\lspan}{\ensuremath{\mathop{\mathrm{span}}}}

\newcommand{\HS}{\mathbb{H}}

\newcommand{\ord}{\ensuremath{{\rm ord}}}
\renewcommand{\diag}{\ensuremath{{\rm diag}}}
\newcommand{\Qpur}{\ensuremath{\QQ_p^{\rm ur}}}
\newcommand{\fp}{\frakp}
\newcommand{\Op}{\ensuremath{\cO_\fp}}
\newcommand{\rhop}{\ensuremath{\varrho^{(g)}_{\diag,\fp}}}

\title{Sturm Bounds for Siegel Modular Forms}
\author{
Olav K. Richter
and
Martin Westerholt-Raum%
\footnote{The first author was partially supported by Simons Foundation Grant $\#200765$. The second author thanks the Max Planck Institute for Mathematics for their hospitality. The paper was partially written, while the second author was supported by the ETH Zurich Postdoctoral Fellowship Program and by the Marie Curie Actions for People COFUND Program.}
}
\newcommand{\headertitle}{Sturm Bounds}
\newcommand{\headerauthors}{O.~K.~Richter, M.~W.-Raum}

\begin{document}
\begin{plainfootnotes}
\maketitle
\end{plainfootnotes}
\thispagestyle{scrplain}

\begin{abstract}
\noindent
We establish Sturm bounds for degree~$g$ Siegel modular forms modulo a prime $p$, which are vital for explicit computations. Our inductive proof exploits Fourier-Jacobi expansions of Siegel modular forms and properties of specializations of Jacobi forms to torsion points. In particular, our approach is completely different from the proofs of the previously known cases $g=1,2$, which do not extend to the case of general~$g$.

\vspace{.25em}
\noindent{\sffamily\tbf MSC 2010: Primary 11F46; Secondary 11F33}
\end{abstract}

\vspace*{2em}


\Needspace*{4em}
\lettrine[lines=2,nindent=.5em]{L}{et} $p$ be a prime.  A celebrated theorem of Sturm~\cite{Sturm-LNM1987} implies that an elliptic modular form with $p$-integral rational Fourier series coefficients is determined by its ``first few" Fourier series coefficients modulo $p$.  Sturm's theorem is an important tool in the theory of modular forms (for example, see~\cite{Ono, Stein} for some of its applications).  Poor and Yuen~\cite{P-Y-paramodular} (and later \cite{C-C-K-Acta2013} for $p\geq 5$) proved a Sturm theorem for Siegel modular forms of degree~$2$. Their work has been applied in different contexts, and for example, it allowed~\cite{C-C-R-Siegel, D-R-Ramanujan} to confirm Ramanujan-type congruences for specific Siegel modular forms of degree $2$.  In~\cite{RR-MRL}, we gave a characterization of $U(p)$ congruences of Siegel modular forms of arbitrary degree, but (lacking a Sturm theorem) we could only discuss one explicit example that occurred as a Duke-Imamo\u glu-Ikeda lift.  If a Siegel modular form does not arise as a lift, then one needs a Sturm theorem to justify its $U(p)$ congruences.

In this paper, we provide such a Sturm theorem for Siegel modular forms of degree~$g\geq 2$.  Our proof is totally different from the proofs of the cases $g=1,2$ in~\cite{Sturm-LNM1987, P-Y-paramodular, C-C-K-Acta2013}, which do not have visible extensions to the case~$g>2$.  More precisely, we perform an induction on the degree~$g$.   As in~\cite{Bru-WR-Fourier-Jacobi}, we employ Fourier-Jacobi expansions of Siegel modular forms, and we study vanishing orders of Jacobi forms.  However, in contrast to~\cite{Bru-WR-Fourier-Jacobi} we consider restrictions of Jacobi forms to torsion points (instead of their theta decompositions), which allow us to relate mod~$p$ diagonal vanishing orders (defined in Section~\ref{sec: preliminaries}) of Jacobi forms and Siegel modular forms.  We deduce the following theorem.

\begin{maintheorem}
\label{thm:maintheorem}
Let $F$ be a Siegel modular form of degree~$g\geq 2$, weight~$k$, and with $p$-integral rational Fourier series coefficients~$c(T)$.
Suppose that
\begin{gather*}
  c(T) \equiv 0 \pmod{p}
\quad
 \text{for all $T=(t_{i\!j})$ with}
\quad
  t_{i\!i} \le \Big(\frac{4}{3}\Big)^g \frac{k}{16}
\text{.}
\end{gather*}
Then $c(T) \equiv 0 \pmod{p}$ for all $T$.
\end{maintheorem}

If a Siegel modular form arises as a lift, then one can sometimes infer that it has integral Fourier series coefficients (see~\cite{P-R-Y-BullAust09}).  The situation is more complicated for Siegel modular forms that are not lifts.  However, if the ``first few diagonal" coefficients of a Siegel modular form are integral (or $p$-integral rational), then Theorem~\ref{thm:maintheorem} implies that all of its Fourier series coefficients are integral (or $p$-integral rational). 

\begin{maincorollary}
\label{cor:maincorollary}
Let $F$ be a Siegel modular form of degree~$g\geq 2$, weight~$k$, and with rational Fourier series coefficients $c(T)$. Suppose that
\begin{gather}
\label{eq:maincorollary-assumption}
  c(T) \in \ZZ
\quad
\text{for all $T=(t_{i\!j})$ with}
\quad
  t_{i\!i} \le \Big(\frac{4}{3}\Big)^g \frac{k}{16}
\text{.}
\end{gather}
Then $c(T) \in\ZZ$ for all $T$.
\end{maincorollary}

\begin{mainremarksenumerate}
\item
Theorem~\ref{thm:maintheorem} and Corollary~\ref{cor:maincorollary} are effective for explicit calculations with Siegel modular forms, since only finitely many $T$ satisfy the condition $t_{i\!i} \le (\frac{4}{3})^g \frac{k}{16}$ for all~$i$.

\item
If $p\geq 5$, then Theorem~\ref{thm:slope-bounds} shows that the bounds $(\frac{4}{3})^g \frac{k}{16}$ in Theorem~\ref{thm:maintheorem}  and in Corollary~\ref{cor:maincorollary} can be replaced by the slightly better bounds $(\frac{4}{3})^g \frac{9k}{160}$. 

\item
If \eqref{eq:maincorollary-assumption} in~Corollary~\ref{cor:maincorollary} is replaced by the assumption that $c(T)$ is $p$-integral rational for all $T=(t_{i\!j})$ with $ t_{i\!i} \le (\frac{4}{3})^g \frac{k}{16}$, then considering the case $q=p$ in the proof of Corollary~\ref{cor:maincorollary} yields that $c(T)$ is $p$-integral rational for all $T$.

\item
One can remove the assumption that $c(T) \in \QQ$ in~Corollary~\ref{cor:maincorollary}.
More precisely, if $F$ is a Siegel modular form of degree~$g\geq 2$, weight~$k$, and with Fourier series coefficients $c(T)\in\CC$ such that \eqref{eq:maincorollary-assumption} holds, then results of~\cite{Chai-Fal} show that $F$ is a linear combination of Siegel modular forms of degree~$g\geq 2$, weight~$k$, and with rational Fourier series coefficients, and applying Corollary~\ref{cor:maincorollary} yields that $c(T) \in \ZZ$ for all $T$. 
\end{mainremarksenumerate}

The paper is organized as follows.  In Section~\ref{sec: preliminaries}, we give some background on Jacobi forms and Siegel modular forms.  In Section~\ref{sec: Vanishing orders of Jacobi forms}, we explore diagonal vanishing orders of Jacobi forms and of their specializations to torsion points. In Section~\ref{sec: Slope bounds for Siegel modular forms}, we inductively establish diagonal slope bounds for Siegel modular forms of arbitrary degree, and we prove Theorem~\ref{thm:maintheorem} and Corollary~\ref{cor:maincorollary}.

\vspace{1ex}


\section{Preliminaries}
\label{sec: preliminaries}

Throughout, $g, k, m \geq 1$ are integers, and $p$ is a rational prime.  We work over the maximal unramified extension $\Qpur$ of $\QQ_p$.  Note that $\Qpur$ contains all $N$\thdash\ roots of unity if $N$ and $p$ are relatively prime. We always write $\fp$ to denote a prime ideal in $\Qpur$, and $\Op$ stands for the localization of $\Qpur$ at~$\fp$. Moreover, we refer to the elements of the local ring $\ZZ_p\cap\QQ$ as $p$-integral rational numbers.

Finally, let $\HS_g$ be the Siegel upper half space of degree~$g$, $\Sp{g}(\ZZ)$ be the symplectic group of degree~$g$ over the integers, and $\rho$ be a representation of $\Sp{g}(\ZZ)$  with representation space $V(\rho)$, and such that $\big[ \ker\rho : \Sp{g}(\ZZ) \big]<\infty$.

\subsection{Siegel modular forms}
\label{sec: Siegel modular forms}

Let $\rmM^{(g)}_k (\rho)$ denote the vector space of Siegel modular forms of degree~$g$, weight~$k$, type $\rho$, and with coefficients in~$\Op$ (see \cite{Sh-Acta78}). If $\rho$ is trivial, then we simply write $\rmM^{(g)}_k$. Recall that an element $F\in\rmM^{(g)}_k(\rho)$ is a holomorphic function $F:\HS_g\ra V(\rho)$ with transformation law
\begin{gather*}
  F\big( (A Z + B) (C Z + D)^{-1} \big)
=
  \rho(M)\,\det(C Z + D)^{k}\, F(Z)
\end{gather*}
for all $M=\left(\begin{smallmatrix} A & B \\ C & D \end{smallmatrix}\right) \in \Sp{g}(\ZZ)$.  Furthermore, $F$ has a Fourier series expansion of the form
\begin{gather*}
 F(Z)= \sum_{T=\rT{T}\geq 0} \hspace{-.5em}
  c(T)\, e^{2\pi i\, \tr(T Z) }
\text{,}
\end{gather*}
where $\tr$ denotes the trace, $\rT T$ is the transpose of $T$, and where the sum is over symmetric, positive semi-definite, and rational $g\times g$ matrices $T$.

If $F\in\rmM^{(g)}_k (\rho)$ such that $F\not\equiv 0\pmod{\,\fp}$, i.e., if there exists a Fourier series coefficient $c(T)$ of $F$ such that $c(T)\not\equiv 0\pmod{\,\fp}$, then the mod~$\fp$ diagonal vanishing order of $F$ is defined by
\begin{gather}
\label{eq:def:Siegel-diagonal-vanishing-order}
  \ord_{\fp}\,F
:=
  \max \big\{0 \le l \in \ZZ \,:\, \forall T=(t_{i\!j}), t_{i\!i} \le l\;\tx{ for all } 1 \le i \le g
             \,:\, c(T) \equiv 0 \pmod{\,\fp} \big\}
\tx{.}
\end{gather}
If $F$ has $p$-integral rational coefficients such that $F\not\equiv 0\pmod{p}$, then $\ord_{p}\,F$ is defined likewise.  Finally, the mod~$\fp$ diagonal slope bound for degree~$g$ (scalar-valued) Siegel modular forms is given by
\begin{gather}
\label{eq:def:diagonal-slope-bound}
  \rho_{\diag,\,\fp}^{(g)}
:=
  \inf_{k}
  \inf_{\substack{F \in \rmM^{(g)}_{k} \\ F \not\equiv 0 \pmod{\,\fp}}}\,
  \frac{k}{\ord_{\fp}\,F}
\tx{,}
\end{gather}
and the definition of the mod~$p$ diagonal slope bound  $\rho_{\diag, p}^{(g)}$ for degree~$g$ (scalar-valued) Siegel modular forms with $p$-integral rational coefficients is completely analogous.

\subsection{Jacobi forms}
\label{sec: Jacobi forms}

Ziegler~\cite{Zi} introduced Jacobi forms of higher degree (extending~\cite{EZ}). Let $\rmJ_{k, m}^{(g)}(\rho)$ denote the ring of Jacobi forms of degree~$g$, weight~$k$, index~$m$, type $\rho$, and with coefficients in~$\Op$.  If $\rho$ is trivial, then we suppress it from the notation.  Recall that Jacobi forms occur as Fourier-Jacobi coefficients of Siegel modular forms: Let $F\in\rmM^{(g+1)}_k (\rho)$, and write $Z =\left(\begin{smallmatrix} \tau & \rT z\\ z & \tau' \end{smallmatrix}\right)\in\HS_{g+1}$, 
where $\tau \in \HS_{g}$, $z \in \CC^{g}$ is a row vector, and $\tau'\in \HS_{1}$ to find the Fourier-Jacobi expansion:
\begin{gather*}
  F(Z)
=
  F(\tau,z,\tau')
=
  \sum_{0 \le m \in \ZZ}
  \phi_m(\tau,z)\, e^{2\pi i m \tau'}
\text{,}
\end{gather*}
where $\phi_m\in\rmJ_{k, m}^{(g)}(\rho)$. We now briefly recollect some defining properties of such Jacobi forms.

Let $G^\rmJ:=\Sp{g}(\RR)\ltimes (\RR^{2g}\tilde{\times}\RR)$ be the real Jacobi group of degree $g$ (see~\cite{Zi}) with group law
\begin{gather*}
[M,(\lambda,\mu),\kappa]\cdot[M',(\lambda',\mu'),\kappa']:=[MM',(\tilde{\lambda}+\lambda', \tilde{\mu}+\mu'),\kappa+\kappa'+\tilde{\lambda} \rT{\mu'}-\tilde{\mu}\rT{\lambda'}]
\text{,}
\end{gather*}
where $(\tilde{\lambda},\tilde{\mu}):=(\lambda, \mu)M'$. For fixed $k$ and $m$, define the following slash operator on functions $\phi:\HS_g \times\CC^g\rightarrow V(\rho)$\,:
\begin{align}
\label{Jacobi-slash}
&
  \Big(\phi\,\big|_{k,m}
  \left[\left(\begin{smallmatrix}A & B\\C & D\end{smallmatrix}\right),\,
  (\lambda, \mu),x\right]
  \Big)(\tau,z)
\;:=\;
  \rho^{-1} \left(\begin{smallmatrix}A & B\\C & D\end{smallmatrix}\right)\,
  \det(C\tau+D)^{-k}
\\\nonumber
&
  \quad\cdot\,
  \exp\Big( 2\pi im\big(
  - (C\tau+D)^{-1} (z+\lambda\tau+\mu) \,C\, \rT(z+\lambda\tau+\mu)\;
  + \lambda \tau \rT\lambda
  + 2\lambda\rT z
  +\mu\rT\lambda
  +x \big) \Big)
\\\nonumber
&
  \quad\cdot\,
  \phi\big( (A\tau+B)(C\tau+D)^{-1},\, (z+\lambda\tau+\mu)(C\tau+D)^{-1} \big)
\end{align}
for all $\left[\left(\begin{smallmatrix}A & B\\C & D\end{smallmatrix}\right), (\lambda, \mu),x\right]\in G^\rmJ$.  A Jacobi form of degree~$g$, weight~$k$, and index~$m$ is invariant under \eqref{Jacobi-slash} when restricted to $\left(\begin{smallmatrix}A & B\\C & D\end{smallmatrix}\right)\in\Sp{g}(\ZZ)$, $(\lambda, \mu)\in\ZZ^{2g}$, and $\kappa=0$.  Moreover, every $\phi\in\rmJ_{k, m}^{(g)}(\rho)$ has a Fourier series expansion of the form
\begin{gather*}
  \phi(\tau,z)
=
  \sum_{T,R}
  c(T,R)\, e^{2\pi i\, \tr(T \tau + zR)}
	\text{,}
\end{gather*}
where the sum is over symmetric, positive semi-definite, and rational $g\times g$ matrices~$T$ and over column vectors $R\in\QQ^g$ such that $4mT - R \rT R$ is positive semi-definite.

Finally, we state the analog of~\eqref{eq:def:Siegel-diagonal-vanishing-order} for Jacobi forms.  Let $\phi\in\rmJ_{k, m}^{(g)}(\rho)$ such that $\phi\not\equiv 0\pmod{\,\fp}$, i.e., there exists a Fourier series coefficient $c(T,R)$ of $\phi$ such that $c(T,R)\not\equiv 0\pmod{\,\fp}$.  Then the mod~$\fp$ diagonal vanishing order of $\phi$ is defined by
\begin{gather}
\label{eq:def:Jacobidiagonal-vanishing-order}
  \ord_{\fp}\,\phi
:=
  \max \big\{0 \le l \in \ZZ \,:\, \forall R, T=(t_{i\!\!j}), t_{i\!i} \le l\;\tx{ for all } 1 \le i \le g
             \,:\, c(T,R) \equiv 0 \pmod{\,\fp} \big\}
\tx{,}
\end{gather}
and if $\phi$ has $p$-integral rational coefficients such that $\phi\not\equiv 0\pmod{p}$, then one defines $\ord_{p}\,\phi$ in the same way.

\vspace{1ex}

\section{Vanishing orders of Jacobi forms}
\label{sec: Vanishing orders of Jacobi forms}

In this section, we discuss diagonal vanishing orders of Jacobi forms and of their evaluations at torsion points.

Throughout, $N$ is a positive integer that is not divisible by $p$.  Consider the $\CC$ vector space
\begin{gather}
  V\big( \rho_{[N]} \big)
:=
  \CC\Big[\big(\tfrac{1}{N} \ZZ^g \slashdiv N \ZZ^g \big)^2 \Big]
=
  \lspan_\CC\big\{ \frake_{\alpha, \beta} \,:\,
                   \alpha, \beta \in \tfrac{1}{N} \ZZ^g / N \ZZ^g 
  \big\}
\tx{,}
\end{gather}
and the representation~$\rho_{[N]}$ on $V\big( \rho_{[N]} \big)$, which is defined by the action of $\Sp{g}(\ZZ)$ on $(\tfrac{1}{N} \ZZ^g / N \ZZ^g)^2$:
\begin{gather}
  \rho_{[N]}\big( M^{-1} \big)\, \frake_{\alpha, \beta}
:=
  \frake_{\alpha', \beta'}
\text{,}
\qquad
  \text{where}\;\;
  \big( \alpha',\beta' \big)
:=
  \big( \alpha,\beta \big)\, M
\quad
\text{for $M\in\Sp{g}(\ZZ)$}
\tx{.}
\end{gather}
If $\phi \in \rmJ^{(g)}_{k,m}$, then $\phi[N]$ is its restriction to torsion points of denominator at most~$N$, i.e., 
\begin{align}
\label{eq:def:jacobi-forms-restriction-to-torsion-points}
\nonumber
  \phi[N]
&:\,
  \HS^{(g)} \lra V\big( \rho_{[N]} \big)
\\
  \phi[N](\tau)
&:=
  \Big(
  \big( \phi \big|_{k, m} [ I_g, (\alpha, \beta), 0 ] \big) (\tau, 0)
  \Big)_{\alpha, \beta \in \frac{1}{N}\ZZ^g / N \ZZ^g }
\tx{,}
\end{align}
where $I_g$ stands for the $g\times g$ identity matrix.  It is easy to see that $\phi[N]$ is a vector-valued Siegel modular form (see also Theorem ~1.3~of~\cite{EZ} and Theorem~1.5 of~\cite{Zi}):

\begin{lemma}
Let $\phi \in \rmJ^{(g)}_{k,m}$. Then $\phi[N] \in \rmM^{(g)}_{k}(\rho_{[N]})$.
\end{lemma}
\begin{proof}
We first argue that $\phi[N]$ is well-defined: If $a, b \in \ZZ^g$, then 
\begin{gather*}
  \phi \big|_{k, m} [ I_g, (\alpha + N a, \beta + N b ), 0 ]
=
  \phi
  \big|_{k, m} [ I_g, (N a, N b), N \alpha \rT b - N \beta \rT a]
  \big|_{k, m} [ I_g, (\alpha, \beta), 0 ]
\text{.}
\end{gather*}
Note that $\kappa := N \alpha \rT{b} - N \beta \rT{a} \in \ZZ$ does not contribute to the action, and we find that the defining expression for $\phi[N]$ is independent of the choice of representatives of $\alpha, \beta \in \frac{1}{N} \ZZ^g \slashdiv N \ZZ^g$.

Next we verify the behavior under modular transformation of~$\phi[N]$. Let $M \in \Sp{g}(\ZZ)$.  Then
\begin{gather*}
  [ I_g, (\alpha, \beta ), 0 ] \cdot [M, (0,0), 0)]
=
  [M, (0,0), 0]\cdot [ I_g, (\alpha', \beta'), 0]
\end{gather*}
with $\big(\alpha', \beta' \big)=\big(\alpha, \beta \big)\, M$, which implies that
\begin{align*}
  \big( \phi[N]_{\alpha, \beta} \big) \big|_k\, M
& =
  \big( \phi \big|_{k, m} [ I_g, (\alpha, \beta), 0 ] \big) (\,\cdot\,, 0) \big)
  \big|_k\, M
=
  \big( \phi \big|_{k, m} [M, (0,0), 0] \cdot [ I_g, (\alpha', \beta'), 0] \big) (\,\cdot\,, 0)\\
& =
  \big( \phi[N]_{\alpha', \beta'} \big)
\text{.}
\end{align*}
\end{proof}

The next lemma relates the mod~$\,\fp$ diagonal vanishing orders of a Jacobi form $\phi$ and its specialization $\phi[N]$.

\begin{lemma}
\label{la:vanishing-order-of-jaocbi-specializations}
Let $\phi \in \rmJ^{(g)}_{k,m}$.  Then $\ord_{\fp}\,\phi[N]\geq \ord_{\fp}\,\phi - \frac{m}{4}$.
\end{lemma}
\begin{proof}
Let $\phi(\tau,z)=\sum_{T,R} c(T,R)\, e^{2\pi i\, (\tr(T \tau)+zR)}$.  Then $\phi[N](\tau)$ equals
\begin{gather}
\label{eq:jacobi-forms-restriction-to-torsion-points:fourier-coefficients}
\begin{aligned}
  \big( \phi\big|_k [I_g, (\alpha, \beta), 0] \big) (\tau, 0)
&=
  e^{2\pi im (\alpha \tau \rT\alpha +\beta\rT\alpha)}\,
  \sum_{T,R}
  c(T,R)\, e^{2\pi i\, \big(\tr(T \tau)+ (\alpha\tau + \beta)R\big)}
\\
&=
  e^{2\pi im\,\beta \rT\alpha}\,
  \sum_{T,R}
  c(T,R) e^{2\pi i\,\beta R}\,
  e^{2\pi i\,\tr\Big(
  \big(T - \frac{1}{4m}R\rT R
       \,+\,
       \frac{1}{m}\rT\big(m \alpha +\frac{1}{2}\rT R\big)\,\big(m \alpha + \frac{1}{2}\rT R\big) 
  \big) \tau \Big)}
\tx{.}
\end{aligned}
\end{gather}
Observe that $c(T,R) e^{2\pi i\,\beta (\rT \alpha +R)}\in\Op$.  It suffices to show that $c(T,R)$ vanishes mod~$\,\fp$ if the diagonal entries $t'_{i\!i}$ of $T':= T - \frac{1}{4m}R\rT R$ are less than $\ord_{\fp}\,\phi- \frac{m}{4}$.

Consider $T, R$ such that $t'_{i\!i} \le \ord_\frakp\,\phi - \frac{m}{4}$ for some fixed~$i$. Note that $c(T, R)$ remains unchanged when replacing $T\mapsto T + \frac{1}{2}(R \lambda+\rT\lambda\rT R) +m \rT \lambda \lambda$ and $R\mapsto R + 2m \rT \lambda$, which corresponds to the invariance of $\phi$ under~$\big|_{k,m}\, [I_g, (\lambda,0) ,0]$. Hence we only have to consider the case of $R = \rT(r_1, \ldots, r_g)$ with $-m \le r_i \le m$.  In this case, $t'_{i\!i} = t_{i\!i} - \frac{1}{4m}r_i^2 \le \ord_{\fp}\,\phi- \frac{m}{4}$ implies that $t_{i\!i} \le \ord_{\fp}\,\phi$, i.e., $c(T, R) \equiv 0 \pmod{\,\fp}$.
\end{proof}

The following lemma associates the mod~$\,\fp$ diagonal vanishing orders of scalar-valued and vector-valued Siegel modular forms.

\begin{lemma}
\label{la:slope-bound-vector-valued}
Suppose that there exists a mod~$\,\fp$ diagonal slope bound $\rhop$ for degree~$g \ge 1$.  Let $\rho$ be a representation of~$\Sp{g}(\ZZ)$ defined over~$\Op$, and assume that its dual $\rho^\ast$ is also defined over~$\Op$.  If $F \in \rmM^{(g)}_k(\rho)$ such that $\ord_{\fp}\,F>k \big\slash \rhop$, then $F \equiv 0 \pmod{\,\fp}$.
\end{lemma}

\begin{proof}
Let $v$ be a linear form on $V(\rho)$, i.e., $v \in V(\rho)^\ast (\Op)$.  Then $\langle F, v \rangle := v \circ F$ is a scalar-valued Siegel modular form of weight $k$ for the group $\ker \rho$.  We obtain a scalar-valued Siegel modular form for the full group $\Sp{g}(\ZZ)$ via the standard construction (see also the proof of Proposition~1.4 of~\cite{Bru-WR-Fourier-Jacobi})
\begin{gather*}
  F_v
:=
  \prod_{M :\, \ker \rho \backslash \Sp{g}(\ZZ)} \langle F, v \rangle |_k\, M
=
  \prod_{M :\, \ker \rho \backslash \Sp{g}(\ZZ)} \langle F, \rho^\ast(M) v \rangle\in\rmM^{(g)}_{dk}
\text{,}
\end{gather*}
where $d := \big[ \ker\rho : \Sp{g}(\ZZ) \big]$.  Observe that $\rho^\ast(M) v \in V(\rho)^\ast (\Op)$, and hence the Fourier series coefficients of $F_v$ do belong to $\Op$.  The assumption $\ord_{\fp}\,F>k \big\slash \rhop$ implies that $\ord_{\fp}\,F_v> dk \big\slash \rhop$, and since $F_v$ is of weight $d k$, we find that $F_v \equiv 0 \pmod{\,\fp}$ for all $v$.  Hence $\langle F, v \rangle$ vanishes mod~$\,\fp$ for every $v$, which proves that $F \equiv 0 \pmod{\,\fp}$.
\end{proof}

The final result in this Section on the mod~$\,\fp$ diagonal vanishing orders of scalar-valued Jacobi forms and Siegel modular forms is an important ingredient in the proof of Theorem~\ref{thm:maintheorem} in the next Section.

\begin{proposition}
\label{prop:slope-bound-jacobi}
Suppose that there exists a mod~$\,\fp$ diagonal slope bound $\rhop$ for degree~$g \ge 1$.  Let $\phi \in \rmJ^{(g)}_{k,m}$ such that $\ord_{\fp}\,\phi>\frac{m}{4} + k \big\slash \rhop$.  Then $\phi\equiv 0\pmod{\,\fp}$.
\end{proposition}
\begin{proof}
Let $\phi(\tau,z)=\sum_{T,R} c(T,R)\, e^{2\pi i\, (\tr(T \tau)+zR)}$. Lemmata~\ref{la:vanishing-order-of-jaocbi-specializations} and~\ref{la:slope-bound-vector-valued} imply that $\phi[N]\equiv 0\pmod{\,\fp}$ for all~$N$ that are relatively prime to $p$.  We prove by induction on the diagonal entries $(t_{i\!i})$ of $T$ that $c(T,R) \equiv 0 \pmod{\,\fp}$.  The constant Fourier series coefficient of $\phi[1]$ equals $c(0,0)$.  Hence $c(0,0) \equiv 0 \pmod{\,\fp}$, i.e., the base case holds.  Next, let $T$ be positive semi-definite and suppose that $c(T', R) \equiv 0 \pmod{\,\fp}$ for all $T'=(t'_{i\!j})$ with $t'_{i\!i} < t_{i\!i}$ for all~$i$. If $R = \rT(r_1, \ldots, r_g)$ such that $|r_i| > m$ for some~$i$, then (as in the proof of Lemma~\ref{la:vanishing-order-of-jaocbi-specializations}) use the modular invariance of $\phi$ to relate $c(T,R)$ to some $c(T',R')$ with $t'_{i\!i} < t_{i\!i}$.  That is, it suffices to show that $c(T, R) \equiv 0 \pmod{\,\fp}$ for $R$ with $-m \le r_i \le m$ for all~$i$.  Now, fix a prime~$N\not=p$ such that $2m < N - 2$. If $\beta= \rT(\beta_1, \ldots, \beta_g) \in \frac{1}{N} \ZZ^g$, then $\phi[N]\equiv 0\pmod{\,\fp}$ implies that (see also~\eqref{eq:jacobi-forms-restriction-to-torsion-points:fourier-coefficients}) 
\begin{gather*}
  \sum_{\substack{R\\ |r_i| \le \frac{N - 1}{2}}} c(T,R) e^{2\pi i\,\beta R}
\equiv
\sum_{R} c(T,R) e^{2\pi i\,\beta R}
  \equiv 0\pmod{\,\fp}
\text{,}
\end{gather*}
where the first congruence follows from the induction hypothesis and the assumption that $2m < N - 2$ (see also the proof of Lemma~\ref{la:vanishing-order-of-jaocbi-specializations}).  Note that $e^{2\pi i\,\beta R}$ are integers in the $N$\nbd th cyclotomic field. Moreover, if
\begin{gather*}
  A:=\big( e^{2\pi i\,\beta R} \big)
  _{\substack{R \in \ZZ^g,\, \frac{1-N}{2} < r_i \le \frac{N-1}{2}\\
              \beta \in \frac{1}{N}\ZZ^g,\, 0 \le N \beta_i \le N-2}}
\text{,}
\end{gather*}
then (observing that $N$ is prime) $\det A=(-1)^{N-1}N^{N-2}$ is the discriminant of the $N$\nbd th cyclotomic field.  In particular, $\det A \not\equiv 0 \pmod{\,\fp}$, and we conclude that $c(T,R) \equiv 0 \pmod{\,\fp}$.  
\end{proof}

\vspace{1ex}

\section{Slope bounds for Siegel modular forms}
\label{sec: Slope bounds for Siegel modular forms}

We prove by induction that there exists a diagonal slope bound $\rhop$ for Siegel modular forms of degree~$g\geq 1$, which then yields Theorem~\ref{thm:maintheorem} and Corollary~\ref{cor:maincorollary}.

\begin{proposition}
\label{prop:relative-sturm-bound}
If $\varrho^{(g-1)}_{\diag,\,\fp}$ is a diagonal slope bound for degree~$g-1$ Siegel modular forms, then $\rhop:= \frac{3}{4} \varrho^{(g-1)}_{\diag,\,\fp}$ is a diagonal slope bound for degree $g$ Siegel modular forms.
\end{proposition}
\begin{proof}
Suppose that there exists an $0 \not\equiv F \in \rmM^{(g)}_k$ whose diagonal slope modulo~$\fp$ is less than $\rhop= \frac{3}{4} \varrho^{(g-1)}_{\diag,\,\fp}$, i.e., the diagonal vanishing order of $F$ is greater than $k \big\slash \rhop$. Consider Fourier-Jacobi coefficients $0\not\equiv\phi_m\in \rmJ^{(g-1)}_{k,m}$ of~$F$.  If $m \leq k \big\slash \rhop$, then 
\begin{gather*}
 \ord_{\fp}\,\phi_m
>
  \frac{k}{\rhop}
\geq
  \frac{m}{4} + \frac{3}{4} \frac{k}{\rhop}
=
  \frac{m}{4} + \frac{k}{\varrho^{(g-1)}_{\diag, \fp}}
\,\text{,}
\end{gather*}
and Proposition~\ref{prop:slope-bound-jacobi} implies that $\phi_m \equiv 0 \pmod{\,\fp}$.

If $m > k \big\slash \rhop$, then an induction on $m$ shows that $\phi_m \equiv 0 \pmod{\,\fp}$.  More specifically, fix an index $m$ and suppose that $\phi_{m'} \equiv 0 \pmod{\,\fp}$ for all $m' < m$.  Thus, the mod~$\fp$ diagonal vanishing order of $\phi_m$ is at least $m$, and we apply again Proposition~\ref{prop:slope-bound-jacobi} to find that $\phi_m \equiv 0 \pmod{\,\fp}$.  Hence $F\equiv 0 \pmod{\,\fp}$, which yields the claim.
\end{proof}

Proposition~\ref{prop:relative-sturm-bound} holds for any prime ideal $\fp$ in $\Qpur$, and hence also for the rational prime $p$. As a consequence we discover explicit slope bounds, which immediately imply Theorem~\ref{thm:maintheorem}.

\begin{theorem}
\label{thm:slope-bounds}
Let $g\geq 1$.  There exist a diagonal slope bound $\varrho^{(g)}_{\diag, p}$  such that
\begin{gather*}
    \varrho^{(g)}_{\diag, p} 
\ge
  16 \cdot \Big(\frac{3}{4}\Big)^g
\text{.}
\end{gather*}
If, in addition,  $g \ge 2$ and $p \ge 5$, then
\begin{gather*}
   \varrho^{(g)}_{\diag, p} 
\ge
  \frac{160}{9} \cdot \Big(\frac{3}{4}\Big)^g
\text{.}
\end{gather*}
\end{theorem}
\begin{proof}
We apply Proposition~\ref{prop:relative-sturm-bound} to the base case $\varrho^{(1)}_{\diag, p} = 12$ (see~\cite{Sturm-LNM1987}), and if $p\geq 5$, to the base case $\varrho^{(2)}_{\diag, p} = 10$ (see~\cite{C-C-K-Acta2013}).
\end{proof}

\begin{example}
If $p\geq 5$, then for $g=3,4,5,6$ we obtain
\begin{gather*}
  \varrho^{(3)}_{\diag, p} \ge 7.5
\text{,}\quad
  \varrho^{(4)}_{\diag, p} \ge 5.6
\text{,}\quad
  \varrho^{(5)}_{\diag, p} \ge 4.2
\text{,}\quad
  \varrho^{(6)}_{\diag, p} \ge 3.1
\text{.}
\end{gather*}
\end{example}

Finally, we prove Corollary~\ref{cor:maincorollary}.

\begin{proof}[Proof of Corollary~\ref{cor:maincorollary}]
Let $F \in \rmM^{(g)}_{k}$ with rational Fourier series coefficients $c(T)$ such that $c(T)\in\ZZ$ for all $T=(t_{ij})$ with $t_{ii} \le \big( \frac{4}{3} \big)^g \frac{k}{16}$ for all~$i$.  Note that $F$ has bounded denominators (this follows from~\cite{Chai-Fal}), i.e., there exists an $0 < l \in \ZZ$ such that $l F \in \rmM^{(g)}_{k}$ has integral Fourier series coefficients.  Let $l$ be minimal with this property.  We need to show that $l = 1$.  If $l\not=1$, then there exists a prime $q$ such that $q \isdiv l$.  Hence $lc(T)\equiv 0\pmod{q}$  for all $T$ with $t_{ii} \le \big( \frac{4}{3} \big)^g \frac{k}{16}$, and Theorem~\ref{thm:maintheorem} asserts that $l c(T)\equiv 0\pmod{q}$ for all $T$.  This contradicts the minimality of~$l$, and we conclude that $l = 1$.
\end{proof}


\renewbibmacro{in:}{}
\renewcommand{\bibfont}{\normalfont\small\raggedright}
\renewcommand{\baselinestretch}{.8}
\Needspace*{4em}
\printbibliography[heading=bibnumbered]



\addvspace{1em}
\titlerule[0.15em]\addvspace{0.5em}

{\noindent\small
Department of Mathematics,
University of North Texas,
Denton, TX 76203, USA\\
E-mail: \url{richter@unt.edu}\\
Homepage: \url{http://www.math.unt.edu/~richter/}
\\[1.5ex]
Max Planck Institute for Mathematics,
Vivatsgasse~7,
D-53111, Bonn, Germany\\
E-mail: \url{martin@raum-brothers.eu}\\
Homepage: \url{http://raum-brothers.eu/martin}
}

\end{document}
